\documentclass[a4paper]{amsart}
\usepackage[utf8]{inputenc}
\usepackage{algpseudocode}
\usepackage{amssymb, xfrac}
\newtheorem{thm}{Theorem}[section]
\newtheorem{lem}[thm]{Lemma}
\newtheorem{obs}[thm]{Observation}
\newtheorem{prop}[thm]{Proposition}
\newtheorem{alg}[thm]{Algorithm}

\newcommand{\ob}{\overline{b}}
\newcommand{\oc}{\overline{c}}
\newcommand{\ga}{\mathfrak{a}}
\newcommand{\gb}{\mathfrak{b}}
\newcommand{\gc}{\mathfrak{c}}
\renewcommand{\gg}{\mathfrak{g}}
\newcommand{\gh}{\mathfrak{h}}
\newcommand{\gp}{\mathfrak{p}}
\newcommand{\gq}{\mathfrak{q}}
\newcommand{\gu}{\mathfrak{u}}
\newcommand{\gv}{\mathfrak{v}}
\newcommand{\kk}{\Bbbk}
\newcommand{\FF}{\mathbb{F}}
\newcommand{\QQ}{\mathbb{Q}}
\newcommand{\ZZ}{\mathbb{Z}}
\newcommand{\onto}{\twoheadrightarrow}
\newcommand{\iso}{\xrightarrow{\;\sim\;}}
\newcommand{\algcl}[1]{\overline{#1}}

\newcommand{\kkXY}{\kk[X,Y]}
\newcommand{\st}{\;|\ }
\newcommand{\divides}{\mid}
\newcommand{\ideal}{\vartriangleleft}
\newcommand{\Ideal}[1]{\left\langle #1\right\rangle}
\newcommand{\tIdeal}[1]{\bigl\langle #1\bigr\rangle}
\newcommand{\card}[1]{|#1|}

\DeclareMathOperator{\lm}{lm}
\DeclareMathOperator{\ord}{ord}
\DeclareMathOperator{\rad}{rad}
\def\clap#1{\hbox to 0pt{\hss#1\hss}}

\def\mathclap{\mathpalette\mathclapinternal}

\def\mathclapinternal#1#2{\clap{$\mathsurround=0pt#1{#2}$}}
\newenvironment{poc}{\begin{proof}[Proof of correctness]}{\end{proof}}
\newcommand{\term}[1]{\emph{#1}}
\algnewcommand{\LineComment}[1]{\Statex // #1}
\title[Intrinsic factorization of ideals in Dedekind domains]{Intrinsic factorization of ideals\\in Dedekind domains}
\author{Mawunyo Kofi Darkey-Mensah \and Przemysław Koprowski}
\email{mdarkeymensah@gmail.com}
\email{przemyslaw.koprowski@us.edu.pl}
\begin{document}
\begin{abstract}
We present a generalization of a polynomial factorization algorithm that works with ideals in maximal orders of global function fields. The method presented in this paper is intrinsic in the sense that it does not depend on the embedding of the ring of polynomials into the Dedekind domain in question. 
\end{abstract}
\maketitle

\section{Introduction}
Let~$R$ be a Dedekind domain. A fundamental and well known property of Dedekind domains is that every ideal $\ga\ideal R$ has a unique factorization into a product of powers of prime ideals. There are cases when this factorization is algorithmically computable. For instance, if $R = \ZZ_K$ is the ring of algebraic integers (i.e. the integral closure of~$\ZZ$) in some algebraic number field $K = \QQ(\vartheta)$, then a suitable algorithm can be found e.g. in \cite[Algorithm 2.3.22]{Cohen00} or \cite[\S2.2]{GMN13}. The algorithms can be adapted also to global function fields. They depend however on knowing an embedding of the ring of integers (or polynomials) into~$R$. In this paper we discuss the problem of performing the computations intrinsically in the monoid of $R$-ideals without relaying on these embeddings. The procedure of factoring ideals, that we propose, resembles a method of factoring polynomials over finite fields. We show how to generalize known algorithms for polynomial factorization to make them work with ideals in maximal orders of global function fields. The ideal to be factored passes through a three-stage process: radical decomposition, distinct degree factorization and equal degree factorization.  

The algorithms presented in \cite{Cohen00, GMN13} are quite efficient, hence the aim of developing intrinsic methods is not so much to reduce the computation time but rather to construct algorithm that do not dependent on the particular structure of global fields and so have potential to be generalized to other rings. In particular the first step of the process, namely the radical decomposition, can be performed in \emph{any} Dedekind domain in which three elementary operation on ideals are computable. This class of rings include coordinate rings of  smooth, algebraically irreducible curves over a computable, perfect field (see Proposition~\ref{prop_coord_is_computable}). Some early experiments of the authors suggest that the algorithms presented here can be generalized to compute primary decomposition of ideals in affine algebras. This subject need further investigation, though.

The paper is organized as follows. In Section~\ref{sec_radical} we discuss the radical decomposition of ideals, which is an analog of a square-free factorization of polynomials. Given an ideal $\ga\ideal R$, this procedure produces a list of radical ideals $\gg_1, \dotsc, \gg_m$ such that~$\ga$ is a product of their respective powers. Next, in Section~\ref{sec_distinct} we show how to factor a radical ideal (i.e. any of the ideals $\gg_1, \dotsc, \gg_m$) into a product of (radical) ideals such that each one of these new ideals is a product of primes of the same residual degree. Finally in Section~\ref{sec_equal} we present a variant of Cantor--Zassenhaus algorithm (Algorithm~\ref{alg_ed_factor}) capable of factoring radical ideals with prime divisors of a fixed degree. The algorithms discussed in this paper were implemented by the authors in a computer algebra system Magma \cite{magma}. In the closing section we presented two examples obtained with our implementation. 

In the whole paper the letter~$R$ always denotes a (fixed) Dedekind domain with a field of fractions~$K$. For readers convenience our notation follows the one used in \cite{AM16}, in particular fraktur letters are used to denote ideals. All the ideals in this paper are integral ideals.

\section{Radical decomposition of ideals}\label{sec_radical}
Let~$R$ be a Dedekind domain and $\ga\ideal R$ be an ideal in~$R$. Assume that~$\ga$ factors into primes as
\begin{equation}\label{eq_factor_a}
\ga = \gp_1^{k_1}\dotsm \gp_s^{k_s},
\end{equation}
where $\gp_1, \dotsc, \gp_s$ are distinct (and unknown) prime ideals and $k_1, \dotsc, k_s > 0$ their multiplicities. Collate the factors of equal multiplicities. For any $j\leq m := \max\{ k_1, \dotsc, k_s \}$ denote
\[
\gg_j := \bigcap_{\substack{1\leq i\leq s\\ k_i = j}} \gp_i.
\]
This way we may write~$\ga$ as a product analogous to a square-free factorization of a polynomial:
\begin{equation}\label{eq_radical_decomp}
\ga = \gg_1\cdot \gg_2^2\dotsm \gg_m^m.
\end{equation}
We shall call~\eqref{eq_radical_decomp} the \term{radical decomposition} of the ideal~$\ga$. The name is justified by the following observation.

\begin{obs}
Ideals $\gg_1, \dotsc, \gg_m$ are radical.
\end{obs}

Indeed, radicals are preserved by intersection (see e.g. \cite[Ch.~1]{AM16}), hence
\[
\rad(\gg_j) 
= \rad\Bigl( \bigcap_{k_i = j}\gp_i \Bigr)
= \bigcap_{k_i = j} \rad(\gp_i)
= \bigcap_{k_i = j} \gp_i
= \gg_j.
\]
In our settings, the ideals $\gg_1, \dotsc, \gg_m$ play roles analogous to square-free factors of a polynomial in case of the square-free factorization, so that we shall call them \emph{radical factors} of~$\ga$.

The following operations are the basic building blocks for our first algorithm:
\begin{itemize}
\item given an ideal~$\ga$ compute its radical $\rad\ga$,
\item given two ideals~$\ga$ and~$\gb$ compute their sum $\ga+\gb$ and the colon ideal $(\ga:\gb) = \{ x\st x\gb\subseteq \ga\}$.
\end{itemize}
We shall say that~$R$ is a ring with \term{computable ideal arithmetic} if all the three operations are computable for ideals of~$R$.

\begin{prop}\label{prop_coord_is_computable}
Let~$\kk$ be a perfect, computable field and $C := \{ F = 0\}$ be a smooth, geometrically irreducible algebraic curve over~$\kk$, defined by a bivariate polynomial $F\in \kkXY$. Then the coordinate ring $R = \kk[C] = \sfrac{\kkXY}{\Ideal{F}}$ admits computable ideal arithmetic.
\end{prop}

The proof of the proposition needs to be preceded by a lemma. Let $\kappa: \kkXY\onto R$ be the canonical epimorphism. By superscripts $\cdot^c$, $\cdot^e$ we shall denote respectively the ideal contraction and extension with respect to~$\kappa$. 

\begin{lem}
Keep the assumptions of the proposition. If $\ga, \gb\ideal R$ are two ideals, then
\[
\ga^{ce} = \ga,\qquad 
\rad(\ga) = \bigl( \rad(\ga^c)\bigr)^e,\qquad
(\ga : \gb) = ( \ga^c : \gb^c)^e.
\]
\end{lem}
 
\begin{proof}
The inclusion $\ga^{ce}\subseteq \ga$ holds always (see e.g. \cite[Proposition 1.17]{AM16}). The other inclusion follows from the fact that~$\kappa$ is an epimorphism. Consequently we have
\[
\rad(\ga) 
= \bigl( \rad(\ga) \bigr)^{ce}
= \bigl( \rad(\ga^c) \bigr)^e,
\]
where the last equality follows from \cite[Exercise~1.18]{AM16}. Likewise we may write
\[
( \ga^c : \gb^c )^e
\subseteq ( \ga^{ce} : \gb^{ce} )
= ( \ga : \gb )
= ( \ga : \gb )^{ce}
\subseteq ( \ga^c : \gb^c )^e.
\]
This concludes the proof.
\end{proof}

\begin{proof}[Proof of Proposition~\ref{prop_coord_is_computable}]
If we do not insist on obtaining the $2$-generators representation of the result, the computation of the sum $\ga+ \gb$ of two ideals can be as simple as a concatenation of their lists of generators. Next, an algorithm for computing a quotient of two ideal in a multivariate polynomial ring is well known and so it follows from the above lemma that one may compute the quotient of ideals in~$R$. Finally, being a Dedekind domain, the ring~$R$ has dimension one. Consequently every nontrivial ideal $\ga\ideal R$ lifts to a zero-dimensional ideal $A\ideal \kkXY$. The radical of a zero-dimensional ideal in a multivariate polynomial ring over a perfect field is computable using Seidenberg's formula (see \cite{Seidenberg74}). Thus, the radical of~$\ga$ is computable, as well by the previous lemma.
\end{proof}

We are now ready to present an algorithm for the radical decomposition. The reader may wish to observe that it is a generalization of Musser's algorithm \cite{Musser71} for the square-free factorization of polynomials over a field of characteristic zero.

\begin{alg}[Radical decomposition of an ideal]\label{alg_radical_decomp}
\mbox{}\\
\textbf{Input:} an ideal~$\ga$ in a Dedekind domain~$R$ with computable ideal arithmetic.\\
\textbf{Output:} radical factors $\gg_1, \dotsc, \gg_m$ of~$\ga$.
\begin{algorithmic}[1]
\LineComment{Initialization}
\State $\ga_0 \gets \ga$
\State $i     \gets 1$
\State $\gb_1 \gets \rad(\ga)$
\State $\ga_1 \gets (\ga_0:\gb_1)$
\LineComment{Main loop}
\While{$\gb_i\neq R$}
	\State $\gb_{i+1} \gets \ga_i + \gb_i$
	\State $\ga_{i+1} \gets (\ga_i : \gb_{i+1})$
	\State $\gg_i     \gets (\gb_i : \gb_{i+1})$
	\State $i \gets i + 1$
\EndWhile
\State \Return $\gg_1, \dotsc, \gg_i$
\end{algorithmic}
\end{alg}

Before we show the correctness of the algorithm, we present a slightly technical lemma that gives an explicit description of ideals~$\ga_i$ and~$\gb_i$ constructed during the execution of the algorithm.

\begin{lem}
Keep the notation as in Algorithm~\ref{alg_radical_decomp}. The ideals~$\ga_i$ and~$\gb_i$ satisfy:
\[
\ga_i = \gg_{i+1}\cdot \gg_{i+2}^2\dotsm \gg_m^{m-i}
\qquad\text{and}\qquad
\gb_i = \gg_i\cdot \gg_{i+1}\dotsm \gg_m.
\]
\end{lem}

\begin{proof}
We proceed by induction. The assertion is trivially true for~$\ga_0$ and~$\gb_1$. Assume that the two formulas hold for ideals~$\ga_{i-1}$ and~$\gb_i$. Take any $x\in \gg_{i+1}\cdot \gg_{i+2}^2\dotsm \gg_m^{m-i}$ and $y\in \gb_i = \gg_i\cdot \gg_{i+1}\dotsm \gg_m$. Then their product~$xy$ lies in $\gg_i\cdot \gg_{i+1}^2\dotsm \gg_m^{m-i+1} = \ga_{i-1}$. Hence $x\in (\ga_{i-1} : \gb_i) = \ga_i$ proving an inclusion $\ga_i \supseteq \gg_{i+1}\cdot \gg_{i+2}^2\dotsm \gg_m^{m-i}$.

Conversely, take $x\in \ga_i = (\ga_{i-1} : \gb_i)$. Fix any prime ideal~$\gp$ dividing~$\ga_{i-1}$ and let $k := \ord_\gp(\ga)$ be the multiplicity of~$\gp$ in the factorization~\eqref{eq_factor_a} of~$\ga$. By the inductive hypothesis, the ideal~$\gp$ divides~$\gb_i$ and $k - i + 1$ is the multiplicity of~$\gp$ in the factorization of~$\ga_{i-1}$. By the strong approximation theorem (see e.g. \cite[Corollary~10.5.11]{Cohn03}) there exists an element $y\in R$ such that
\[
\ord_\gp y = 1\qquad\text{and}\qquad \ord_\gq y \geq 1\quad\text{for all }\gq\divides \gb_i.
\]
In particular, $y$ is an element of~$\gb_i$ and $y\notin \gp^2$. By the definition of the colon ideal, $xy\in x\cdot \gb_i\subseteq \ga_{i-1}\subseteq \gp^{k - i + 1}$. It follows that the $\gp$-adic valuation of the product~$xy$ is at least $k - i + 1$. Therefore, we have
\[
k - i + 1 
\leq \ord_\gp(xy)
= \ord_\gp x + 1.
\]
Consequently $\ord_\gp x\geq k - i$, which means that $x\in \gp^{k - i}$. As this holds for every prime~$\gp$ dividing~$\ga_{i-1}$, we see that
\[
x\in \bigcap_{\substack{\gp\divides \ga\\ \mathclap{\ord_\gp(\ga)\geq i}}} \gp^{\ord_\gp(\ga) - i}
= \prod_{\substack{\gp\divides \ga\\ \mathclap{\ord_\gp(\ga)\geq i}}} \gp^{\ord_\gp(\ga) - i}
= \prod_{k\geq i}\biggl( \prod_{\substack{\gp\divides \ga\\ \mathclap{\ord_\gp(\ga)=k}}} \gp\biggr)^{k - i} 
\hspace{-1em}= R\cdot \gg_{i+1}\cdot \gg_{i+2}^2\dotsm \gg_m^{m-i}.
\]
This shows that $\ga_i \subseteq \gg_{i+1}\cdot \gg_{i+2}^2\dotsm \gg_m^{m-i}$.

We now prove the equality $\gb_{i+1} = \gg_{i+1}\dotsm \gg_m$. One inclusion is immediate. 
\begin{align*}
\gb_{i+1}
&= \ga_i + \gb_i = \Ideal{ \gg_{i+1}\cdot \gg_{i+1}^2\dotsm \gg_m^{m-i}\cup \gg_i\cdot \gg_{i+1}\dotsm \gg_m }\\
\intertext{The radical ideals~$\gg_i$ are pairwise coprime, hence}
&= \Ideal{ \bigcap_{j\geq i+1} \gg_j^{j-i}\cup \Bigl(\gg_i\cap \bigcap_{j\geq i+1}\gg_j\Bigr) }\\
&= \Ideal{ \Bigl( \bigcap_{j\geq i+1} \gg_j^{j-i}\cup \gg_i\Bigr) \cap \Bigl(\bigcap_{j\geq i+1}\gg_j^{j-i}\cup \bigcap_{j\geq i+1}\gg_j\Bigr) }\\
&\subseteq \Ideal{ \Bigl(\bigcap_{j\geq i+1}\gg_j\cup \gg_i\Bigr) \cap  \bigcap_{j\geq i+1}\gg_j}\\
&= \Ideal{ \bigcap_{j\geq i+1}\gg_j} = \gg_{i+1}\cdot \gg_{i+1}\dotsm \gg_m.
\end{align*}
In order to show the other inclusion fix an element $x\in \gg_{i+1}\dotsm \gg_m$. Ideals~$\gg_i$ and $\gg_{i+1}\gg_{i+2}^2\dotsm \gg_m^{m-i} = \ga_i$ are relatively prime, hence there exist elements $y\in \gg_i$ and $z\in \ga_i$ such that $x = y + z$. Therefore, for any $j\geq i+1$ we have
\[
y = x - z \in \gg_j + \ga_i\subseteq \gg_j + \gg_j = \gg_j.
\]
It follows that $y\in \gg_{i+1}\cap\dotsc \cap\; \gg_m = \gb_i$. Consequently, $x = y + z\in \gb_i + \ga_i = \gb_{i+1}$.
\end{proof}

We are now ready to show the correctness of the algorithm.

\begin{proof}[Proof of correctness of Algorithm~\ref{alg_radical_decomp}]
It follows immediately from the preceding lemma that the algorithm terminates. All we need to show is that for every index~$i$ the colon ideal $(\gb_i:\gb_{i+1})$ equals the sought radical ideal~$\gg_i$. One inclusion is immediate. By the lemma we have
\[
\gg_i\cdot \gb_{i+1} = \gg_i\cdot \bigl( \gg_{i+1}\dotsm \gg_m\bigr) = \gb_i
\]
and so $\gg_i\subseteq (\gb_i : \gb_{i+1})$. We need to prove the other inclusion. To this end take any $x\in (\gb_i : \gb_{i+1})$ and fix a prime divisor~$\gp$ of~$\gg_i$. The multiplicity of~$\gp$ in the factorization of~$\ga$ is thus $\ord_\gp(\ga) = i$. By the strong approximation theorem there is an element $y\in R$ such that $y\in \gb_{i+1}\setminus \gp$. Now, $xy \in x\cdot \gb_{i+1}\subseteq \gb_i\subseteq \gp$ but $y\notin \gp$, it follows that $x\in \gp$. This shows that~$x$ belongs to every prime divisor~$\gp$ of~$\ga$ of multiplicity $\ord_\gp(\ga) = i$. Therefore
\[
x\in \bigcap_{\substack{\gp\divides a\\ \mathclap{\ord_\gp(\ga) = i}}}\gp = \gg_i.
\]
This proves the correctness of the algorithm.
\end{proof}

\section{Distinct degree factorization}\label{sec_distinct}
In this and the next section we restrict our attention to maximal orders in global function fields. Thus, let~$\kk$ be a fixed finite field and let $R = \kk[C] = \sfrac{\kkXY}{\Ideal{F}}$ be a Dedekind domain which is a coordinate ring of a smooth and geometrically irreducible curve~$C$. In particular~$R$ is a maximal order in a field~$\kk(C)$ of rational functions on~$C$ (i.e. in a global function field). Given a radical ideal $\ga\ideal R$, consider its factorization into primes
\[
\ga = \gp_1\dotsm \gp_s.
\]
Collate the primes with respect to their residual degrees setting
\[
\gh_j := \prod_{\substack{\gp\divides \ga\\ \deg \gp = j}} \gp.
\]
Consequently the ideal~$\ga$ may be expressed as a product
\begin{equation}\label{eq_dd_factor}
\ga = \gh_1\dotsm \gh_m,\qquad\text{where}\quad m := \max\{\deg\gp\st \gp\text{ divides }\ga\}.
\end{equation}
By analogy to the polynomial case, we shall call~\eqref{eq_dd_factor} the \term{distinct degree factorization} of~$\ga$.

We will compute the distinct degree factorization of a given ideal~$\ga$ by constructing successive greatest common divisors (in the lattice of $R$-ideals) of~$\ga$ and~$\gu_k$, where~$\gu_k$ is the intersection of all primes of residual degrees dividing~$k$:
\[
\gu_k := \prod_{\substack{\gp\text{ prime}\\\deg\gp \divides k}} \gp.
\]
Before we continue, recall that with every prime ideal~$\gp$ one can associate a unique point $(x_\gp, y_\gp)$ on the curve~$C$, with coordinates in the algebraic closure~$\algcl{\kk}$ of~$\kk$. To this end treat elements of $R = \kk[C]$ as polynomial functions on~$C$ and set $(x_\gp, y_\gp)$ to be a unique point where all elements of~$\gp$ vanish simultaneously. In this section~$\kk$ is a finite field, say $\kk = \FF_q$ for some prime power~$q = p^l$. The degree of~$\gp$ divides~$k$ if and only if $x_\gp, y_\gp$ lie in~$\FF_{q^k}$. It is well known that $\FF_{q^k}$ consists of elements satisfying $a^{q^k}-a = 0$. Apply this fact to both coordinates.

\begin{lem}\label{lem_universal_ideal}
For every $k\geq 1$, the ideal $\gu_k$ is generated by $x^{q^k}-x$ and $y^{q^k}-y$, where $x,y$ are images in~$R$ of $X,Y\in \kkXY$.
\end{lem}

\begin{proof}
Fix $k\geq 1$ and denote $\gv_k := \tIdeal{ x^{q^k}-x, y^{q^k}-y}\ideal R$. We shall prove first that the ideal~$\gv_k$ is contained in~$\gu_k$. It suffices to show that both its generators belong to every prime ideal~$\gp$ of~$R$ whose residual degree divides~$k$. Take any such prime~$\gp$. The coordinates $x_\gp, y_\gp$ of the associated point belong to~$\FF_{q^k}$, hence $x_\gp^{q^k}-x_\gp = 0 = y_\gp^{q^k} - y_\gp$. Thus the generators of~$\gv_k$ vanish on $(x_\gp, y_\gp)$ and so $\gv_k\subseteq \gp$.

Next, we show that~$\gv_k$ is not contained in any prime ideal~$\gp$ whose degree does not divide~$k$. Suppose that $\gv_k\subset \gp$ for some prime ideal~$\gp$. In particular the generators $x^{q^k}-x, y^{q^k}-y$ belong to~$\gp$. Therefore, they vanish on the associated point $(x_\gp, y_\gp)$, which means that $x_\gp, y_\gp\in \FF_{q^k}$ and so $\deg \gp$ divides~$k$.

From what we have proved so far it follows that~$\gu_k$ is the radical of~$\gv_k$. In order to conclude the proof, it suffices to show that~$\gv_k$ is a radical ideal itself. To this end we show that for every prime ideal~$\gp$, $\deg \gp\divides k$ the valuation of at least one of the generators of~$\gv_k$ equals~$1$. Consider two (reducible) algebraic curves $C_1 := \bigl\{x^{q^k}= x\bigr\}$ and $C_2 := \bigl\{ y^{q^k}=y\bigr\}$. They both consist of parallel lines but they are not parallel to each other. Suppose that $\ord_\gp (x^{q^k}-x)> 1$ for some~$\gp$. This means that~$\gp$ is a ramified extension of an ideal $p\cdot \FF_q[x]$ for some irreducible polynomial~$p$ in~$x$. We may identify the valuation $\ord_\gp(x^{q^k}-x)$ with the intersection index $I\bigl( (x_\gp,y_\gp), C\cap C_1\bigr)$. If
\[
\ord_\gp\bigl( x^{q^k}-x ) = I\bigl( (x_\gp,y_\gp), C\cap C_1\bigr) > 1,
\]
then~$C$ is tangent to~$C_1$ at $(x_\gp, y_\gp)$. Consequently it cannot be tangent to~$C_2$ at $(x_\gp, y_\gp)$ as~$C$ is non-singular. Therefore
\[
\ord_\gp\bigl( y^{q^k}-y ) = I\bigl( (x_\gp,y_\gp), C\cap C_2\bigr) = 1.
\]
This shows that for every prime ideal~$\gp$, whose residual degree divides~$k$, either $x^{q^k}-x\in \gp\setminus\gp^2$ or $y^{q^k}-y\in \gp\setminus \gp^2$. This implies that~$\gv_k$ is radical.
\end{proof}

We may now present an algorithm for distinct degree factorization.
\begin{alg}[Distinct degree factorization]\label{alg_dd_factor}
\mbox{}\\
\textbf{Input:} a radical ideal~$\ga\ideal R$.\\
\textbf{Output:} distinct degree factors $\gh_1, \dotsc, \gh_m$ of~$\ga$.
\begin{algorithmic}[1]
\LineComment{Initialization}
\State $k     \gets 1$
\State $\ga_1 \gets \ga$
\LineComment{Main loop}
\While{$\ga_k\neq R$}
	\State $\gu_k \gets \tIdeal{ x^{q^k}-x, y^{q^k}-y }$
	\State $\gh_k \gets \gu_k + \ga_k$
	\State $\ga_{k+1} \gets (\ga_k : \gh_k)$
	\State $k \gets k + 1$
\EndWhile
\State \Return $\gh_1, \dotsc, \gh_k$
\end{algorithmic}
\end{alg}

\begin{poc}
We proceed by an induction on~$k$. Assume that $\gh_{k-1}$ is the $(k-1)$-th distinct degree factor of $\ga$ and~$\ga_k$ is the product of the prime divisors of~$\ga$ with residual degrees at least~$k$. This is trivially true for $\ga_1 = \ga$ and $\gh_0 := R$. Lemma~\ref{lem_universal_ideal} asserts that $\tIdeal{ x^{q^k}-x, y^{q^k}-y } = \gu_k$. Compute
\begin{multline*}
\gu_k + \ga_k 
= \Ideal{ \gu_k\cup \ga_k}
= \Ideal{ \bigcap_{\deg\gp\divides k}\gp\ \cup \bigcap_{\substack{\gq\divides\ga\\ \deg\gq\geq k}}\gq }\\
= \Ideal{ \bigcap_{\substack{\deg \gp \divides k\\ \gq\divides \ga\\ \deg \gq\geq k}} (\gp\cup \gq) }
= \Ideal{ \bigcap_{\substack{\deg \gp \divides k\\ \gq\divides \ga\\ \deg \gq\geq k}} (\gp + \gq) }.
\end{multline*}
Now prime ideals $\gp,\gq$ are either equal or relatively prime. Hence $\gp+ \gq = \gp$ when $\gp = \gq$ and $\gp+\gq = R$ if $\gp\neq \gq$. Consequently the above formula simplifies to
\[
\gu_k + \ga_k = \bigcap_{\substack{\gp\divides \ga\\ \deg \gp = k}} \gp = \gh_k.
\]
It follows that $\ga_{k+1} = (\ga_k : \gh_k)$ is the product off all those prime divisors of~$\ga$ that have degrees strictly greater than $k$.
\end{poc}

\section{Equal-degree factorization}\label{sec_equal}
After performing a radical decomposition and distinct degree factorization, we are left with a list of radical ideals such that each one is a product of primes all having the same (known) residual degree. We can deal with such ideals using a generalization of a classical Cantor--Zassenhaus algorithm. We shall first note the following fact.

\begin{lem}
If $\ga\ideal R$ is a nonzero radical ideal, then the number of elements of the residue ring $\sfrac{R}{\ga}$ is algorithmically computable.
\end{lem}

\begin{proof}
As in the proof of Proposition~\ref{prop_coord_is_computable} we use the ideal contraction with respect to the canonical epimorphism $\kappa: \kkXY\onto R$. The ring~$R$ is a Dedekind domain and $\ga\neq \{0\}$, hence $\sfrac{R}{\ga}$ is a finite ring isomorphic to $\sfrac{\kkXY}{\ga^c}$. The number of elements of the latter ring is computable using a well known trick of counting monomials not in $\lm(\ga^c)$, where  $\lm(\ga^c)$ is an ideal spanned by leading monomials of~$\ga^c$ with respect any monomial order in $\kkXY$.
\end{proof}

From now on we assume that $\ga\ideal R$ is a radical ideal with some (unknown) factorization
\[
\ga = \gp_1\dotsm \gp_m
\]
and the residual degrees of $\gp_1, \dotsc, \gp_m$ are all the same and a priori known. Denote this common degree by~$d$.

\begin{lem}
Let~$b$ be an element of~$R$ not in~$\ga$. Denote $\ob := b + \ga$ the class of~$b$ in $\sfrac{R}{\ga}$ and $e := q^d-1$. The following conditions are equivalent:
\begin{enumerate}
\item\label{it_proper_divisor} the ideal $\gb := \Ideal{b} + \ga$ is a proper divisor of~$\ga$;
\item\label{it_zero_divisor} the element $\ob$ is a zero-divisor in $\sfrac{R}{\ga}$;
\item\label{it_exponent} $\ob^{e}\neq 1$.
\end{enumerate}
\end{lem}

\begin{proof}
Assume that~$\gb$ is a proper divisor of~$\ga$. This means that $\ga\varsubsetneq \gb\varsubsetneq R$. In particular~$b$ cannot lie in~$\ga$ and so $\ob\neq 0$. The ring $\sfrac{R}{\ga}$ is finite, hence it suffices to show that~$\ob$ is not invertible. Suppose a contrario that there is an element $c\in R$ such that $\oc\cdot \ob = 1$. But then $1\in \gb$ and this contradicts the assumption that $\gb\neq R$.

The implication $\eqref{it_zero_divisor}\Longrightarrow\eqref{it_exponent}$ is trivial. In order to prove the remaining implication $\eqref{it_exponent}\Longrightarrow\eqref{it_proper_divisor}$, assume that $\ob^e\neq 1$. By the Chinese reminder theorem there is an isomorphism
\[
\varphi : \sfrac{R}{\ga} \iso \sfrac{R}{\gp_1}\times \dotsb \times \sfrac{R}{\gp_m},
\]
where each quotient ring $\sfrac{R}{\gp_i}$ is in turn isomorphic to $\FF_{q^d}$. Let $\pi_i : \sfrac{R}{\gp_1}\times \dotsb \times \sfrac{R}{\gp_m}\onto \sfrac{R}{\gp_i}$ be the projection onto the $i$-th coordinate. For every $i\leq m$, the image $(\pi_i\circ \varphi)(\ob^e)$ is either~$1$ if $b\notin \gp_i$, or~$0$ if $b\in \gp_i$. Not all coordinates can be equal~$1$, because $\ob^e\neq 1$. Neither all the coordinates are equal zero, since $b\notin \ga$. Denote
\[
I := \bigl\{ i\leq m : (\pi\circ \varphi)(b^e) = 0\bigr\} = \bigl\{ i\leq m: b\in \gp_i\},
\]
we then have
\[
\gb = \prod_{i\in I} \gp_i
\]
and it is clear that $\ga\varsubsetneq \gb\varsubsetneq R$.
\end{proof}

We may now present a randomized recursive algorithm, in a spirit of Cantor--Zassenhaus, for factoring radical ideals of constant residual degree.
\begin{alg}[Equal degree factorization]\label{alg_ed_factor}
\mbox{}\\
\textbf{Input:} a radical ideal~$\ga\ideal R$ and an integer~$d$ such that the residual degree of every prime factor of~$\ga$ equals~$d$.\\
\textbf{Output:} prime factors $\gp_1, \dotsc, \gp_m$ of~$\ga$.
\algrenewcommand\algorithmicuntil{\textbf{end repeat}}
\begin{algorithmic}[1]
\LineComment{Recursion termination}
\If{$\card{\sfrac{R}{\ga}} = q^d$}
	\State \Return $\ga$
\EndIf
\LineComment{Main loop}
\Repeat
	\State $b\gets\mbox{}$ random element of $R\setminus\ga$
	\State $\ob \gets b + \ga\in \sfrac{R}{a}$
	\If{$\ob^{q^d-1}\neq 1$}
		\State $\gb \gets \Ideal{b} + \ga$
		\State $\gc \gets ( \ga : \gb )$
		\Statex {\hspace{\algorithmicindent}\hspace{\algorithmicindent}// Recursion}
		\State $r_1 \gets \mbox{}$ Equal degree factorization of~$\gb$
		\State $r_2 \gets \mbox{}$ Equal degree factorization of~$\gc$
		\State \Return $r_1\cup r_2$
	\EndIf
\Until{}
\end{algorithmic}
\end{alg}

The correctness of the algorithm follows immediately from the lemma preceding it. For the sake of completeness we present an algorithm for the complete factorization of an ideal, that summarizes the whole discussion.
\begin{alg}[complete factorization]\label{alg_factorization}
\mbox{}\\
\textbf{Input:} an ideal~$\ga$ in~$R$.\\
\textbf{Output:} the list of pairs $(\gp_i, k_i)$ of prime divisors and multiplicities, see Eq.~\eqref{eq_factor_a}.
\begin{algorithmic}[1]
\State $\mathop{Factors} \gets []$
\State $G\gets\mbox{}$ radical decomposition of~$\ga$ \textup(Algorithm~\ref{alg_radical_decomp}\textup)
\For{$j\leq \card{G}$}
	\State $\gg_j\gets G[j]$
	\State $H \gets\mbox{}$ distinct degree factorization of~$\gg_j$ \textup(Algorithm~\ref{alg_dd_factor}\textup)
	\For{$d\leq \card{H}$}
		\State $\gh_d \gets H[d]$
		\State $P \gets \mbox{}$ equal degree factorization of $\gh_d$ \textup(Algorithm~\ref{alg_ed_factor}\textup)
		\State $\mathop{Factors} \gets \mathop{Factors}\cup \bigl[ (\gp, j) : \gp \in P \bigr]$
	\EndFor
\EndFor
\State \Return $\mathop{Factors}$
\end{algorithmic}
\end{alg}

\section{Examples}\label{sec_examples}
The authors implemented algorithms described in this paper in a computer algebra system Magma~\cite{magma}. Below we preset two examples computed using our implementation.


\subsubsection*{Example} Let $K = \FF_{13}(x, y)$ be a hyperelliptic function field  given by a generating polynomial
\[
F = y^2 - (x^5 - x)(x^4 + 2)
\]
and let $R:=\sfrac{\FF_{13}[x,y]}{\Ideal{F}}$. Consider the ideal $\ga \ideal R$ 
\begin{multline*}
\ga =  \langle 
x^9 + 8x^7 + 5x^6 + 10x^5 + 6x^4 + 4x^3 + 9x^2 + 6x + 4,\\
\qquad 11x^8 + 8x^7 + 2x^6 + 10x^5 + 6x^4 + x^3y + x^3 + 4x^2y + 7x^2 + 4xy + 9y + 7 \rangle
\end{multline*}
Use Algorithm~\ref{alg_radical_decomp} to compute the radical decomposition $\ga = \gg_1\cdot \gg_2^2$, where
\begin{align*}
\gg_1 &= \langle x^6 + 9x^5 + 7x^4 + 10x^3 + 4x^2 + 4x + 12,\\
&\qquad y + 12x^5 + x^4 + 11x^3 + 10x^2 + 3x + 8 \rangle, \\
\gg_2 &= \Ideal{ x^3 + 4x^2 + 4x + 9, y + 7x^2 + 9x + 12 }.\\
\end{align*}
Next, using Algorithm~\ref{alg_dd_factor}, we compute the distinct degree factorization for each element of the radical decomposition. For $\gg_1$ it returns two trivial factors $\gh_{11} = \gh_{12} = R$ and one nontrivial, degree~$3$ factor
\begin{multline*}
\gh_{13} = \langle 8x^5y + 5x^4y + 9x^3y + xy + 5y + 1,\\
x^6y + 9x^5y + 7x^4y + 10x^3y + 4x^2y + 4xy + 12y \rangle.
\end{multline*}
For $\gg_2$ the situation is fully analogous: $\gh_{21}= \gh_{22} = R$ and
\[
\gh_{23} = \Ideal{ 5x^2y + 5xy + 6y + 1, x^3y + 4x^2y + 4xy + 9y }.
\]
Finally we compute the equal degree factorization for each of the above factors using Algorithm~\ref{alg_ed_factor}. For $\gh_{13}$ we obtain the following primes
\begin{align*}
\gp_1 &= \Ideal{ x^3 + 4x^2 + 4x + 9, y + 6x^2 + 4x + 1 }\\
\gp_2 &= \Ideal{  x^3 + 5x^2 + 9x + 10, y + 3x^2 + 7x + 4 }\\
\intertext{and for $\gh_{23}$ we get}
\gp_3 &= \Ideal{ x^3 + 4x^2 + 4x + 9, y + 7x^2 + 9x + 12 }.
\end{align*}
Hence the complete factorization of~$\ga$ is $\gp_1\cdot \gp_2\cdot \gp_3^2$.
	

\subsubsection*{Example}
In this example, we consider an elliptic function field $K = \FF_{19}(x, y)$ with full constant field $\FF_{19}$, where 
\[
y^2 + y = x^3 - 2x^2 + 1
\]
Take a and ideal
\begin{align*}
\ga 
&= \langle x^{21} + 14x^{20} + 9x^{19} + 4x^{18} + 5x^{17} + 12x^{16} + 9x^{15} + 7x^{14} + 12x^{13} + 8x^{12}\\
&\quad + 3x^{11} + 8x^{10} + 14x^9 + 7x^8 + 12x^7 + x^6 + 9x^5 + 13x^4 + 9x^3 + 4x^2 + 18x + 4, \\ 
& x^3y + 6x^2y + 3xy + 17y + 7x^{18} + 7x^{17} + 11x^{16} + x^{15} + 18x^{13} + 8x^{12} + 9x^{11}\\ 
&\quad + 15x^{10} + 13x^9 + 18x^8 + 12x^7 + x^6 + 14x^5 + 10x^4 + 7x^3 + 15x^2 + 9x + 5\rangle.
\end{align*}
We again use Algorithm~\ref{alg_radical_decomp} to factor~$I$ into a product of radical ideals. It returns one trivial factor $\gg_{3} = R$ and three nontrivial factors $\gg_{1}, \gg_{2}$ and $\gg_{4}$ where
\begin{align*}
\gg_1 &= \langle x^3 + 6x^2 + 3x + 17, x^3y + 6x^2y + 3xy + 17y \rangle,\\
\gg_2 &= \Ideal{ x^3 + 4x + 17, y + 8x^2 + 2x + 9 }, \\
\gg_4 &= \Ideal{ x^3 + 2x^2 + 10x + 4, y + 8x^2 + 3x }.
\end{align*}
Now we compute the distinct degree factors of the above ideals. For $\gg_1$ we have two trivial factors $\gh_{11} = \gh_{13} = R$ and two nontrivial one, degrees~$2$ and~$4$, respectively:
\begin{align*}
\gh_{12} &= \Ideal{ x + 1 }.\\
\gh_{14} &= \Ideal{ x^2 + 5x + 17 }
\intertext{For $\gg_2$ it returns two trivial factors $\gh_{21} = \gh_{22} = R$ and one nontrivial, degree~$3$ factor}
\gh_{23} &= \Ideal{x^3 + 4x + 17, y + 8x^2 + 2*x + 9}.
\intertext{Similarly for~$\gg_4$ we have $\gh_{41} = \gh_{42} = R$ and }
\gh_{43} &= \Ideal{x^3 + 2x^2 + 10x + 4, y + 8x^2 + 3x}.
\end{align*}
Finally we use Algorithm~\ref{alg_ed_factor} to compute the equal degree factorization. It turns out that all four ideals $\gh_{12}, \gh_{14}, \gh_{23}$ and~$\gh_{43}$ are in fact prime. Denoting  
\[
\gp_1 := \gh_{12},\quad
\gp_2 := \gh_{14},\quad
\gp_3 := \gh_{23},\quad
\gp_4 := \gh_{43},
\]
we obtain the complete factorization of~$\ga$, namely $\ga = \gp_1\cdot \gp_2\cdot \gp_3^2\cdot \gp_4^3$.

\bibliographystyle{plain} 
\bibliography{radical_decomposition}
\end{document}